\documentclass{article}

\usepackage[a4paper,left=1in,width=12cm]{geometry}

\usepackage{algorithmic}
\usepackage{amsfonts}
\usepackage{amsmath}
\usepackage{graphicx}
\usepackage[round]{natbib}
\usepackage{microtype}

\usepackage{amsthm}
\newtheorem{theorem}{Theorem}[section]
\newtheorem{lemma}[theorem]{Lemma}
\newtheorem{corollary}[theorem]{Corollary}
\newtheorem{prop}[theorem]{Proposition}

\theoremstyle{definition}
\newtheorem{algorithm}[theorem]{Algorithm}
\newtheorem{definition}[theorem]{Definition}

\newtheorem*{notation}{Notation}

\newtheorem*{assumption}{Assumption A}

\usepackage{hyperref}

\def\ast{{\mathstrut*}}
\def\avg{^{(\delta)}}
\DeclareMathOperator{\bias}{bias}
\def\CN{\mathcal{N}}
\def\CO{\mathcal{O}}
\def\cond#1#2{(#1\mskip1mu|\mskip1mu #2)}
\def\condl#1#2{\bigl(#1\!\bigm|\!#2\bigr)}
\DeclareMathOperator{\cost}{cost}
\let\d=\partial
\let\downto=\downarrow
\def\E{\mathbb{E}}
\let\eps=\varepsilon
\def\hYn{Y^{(\delta)}_{\lfloor\hat N p\rfloor}}
\def\hYN{\hat Y^{(\delta)}_{\hat N}}
\def\lhs{\hskip5mm&\hskip-5mm}
\DeclareMathOperator{\MSE}{MSE}
\def\N{\mathbb{N}}
\let\phi=\varphi
\def\quark{{\setbox0\hbox{$x$}\hbox to\wd0{\hss$\cdot$\hss}}}
\def\R{\mathbb{R}}
\DeclareMathOperator{\RMSE}{RMSE}
\DeclareMathOperator{\tr}{tr}
\DeclareMathOperator{\Var}{Var}
\def\Yn{Y^{(\delta)}_n}
\def\Ynn{Y^{(\delta_n)}_n}

\begin{document}

\title{The Rate of Convergence for Approximate Bayesian Computation}
\author{Stuart Barber, Jochen Voss and Mark Webster}
\maketitle

\begin{abstract}
  Approximate Bayesian Computation (ABC) is a popular computational
  method for likelihood-free Bayesian inference.  The term
  ``likelihood-free'' refers to problems where the likelihood is
  intractable to compute or estimate directly, but where it is
  possible to generate simulated data~$X$ relatively easily given a
  candidate set of parameters~$\theta$ simulated from a prior
  distribution.  Parameters which generate simulated data within some
  tolerance~$\delta$ of the observed data~$x^\ast$ are regarded as
  plausible, and a collection of such $\theta$ is used to estimate the
  posterior distribution $\theta\,|\,X\!=\!x^\ast$.  Suitable choice
  of~$\delta$ is vital for ABC methods to return good approximations
  to~$\theta$ in reasonable computational time.

  While ABC methods are widely used in practice, particularly in
  population genetics, rigorous study of the mathematical properties
  of ABC estimators lags behind practical developments of the method.
  We prove that ABC estimates converge to the exact solution under
  very weak assumptions and, under slightly stronger assumptions,
  quantify the rate of this convergence.  In particular, we show that
  the bias of the ABC estimate is asymptotically proportional
  to~$\delta^2$ as $\delta \downto0$.  At the same time, the
  computational cost for generating one ABC sample increases like
  $\delta^{-q}$ where $q$ is the dimension of the observations.  Rates
  of convergence are obtained by optimally balancing the mean squared
  error against the computational cost.  Our results can be used to
  guide the choice of the tolerance parameter~$\delta$.
\end{abstract}

\vskip1\baselineskip

\noindent
\textbf{keywords:} Approximate Bayesian Computation,
likelihood-free inference,
Monte Carlo methods,
convergence of estimators,
rate of convergence

\vskip1\baselineskip

\noindent
\textbf{MSC2010 classes:} \\
62F12 (Asymptotic properties of estimators), \\
62F15 (Bayesian inference), \\
65C05 (Monte Carlo methods)


\section{Introduction}

Approximate Bayesian Computation (ABC) is a popular method for
likelihood-free Bayesian inference.  ABC methods were originally
introduced in population genetics, but are now widely used in
applications as diverse as epidemiology~\citep{TanFraLucSis06,
  BluTra10, WalAllLeeSma10}, materials science~\citep{BorColSis07},
parasitology~\citep{DroPet11}, genetic evolution~\citep{ThoAnd06,
  Fag-etal07, Rat-etal09, WegExc10, Beau10, Wil-etal11}, population
migration~\citep{Gui-etal10} and conservation
studies~\citep{LopBoe10}.

One of the earliest articles about the ABC approach, covering
applications in population genetics, is~\citet{TavBalGriDon97}.  Newer
developments and extensions of the method include placing ABC in an
MCMC context~\citep{Maj-etal03}, sequential ABC~\citep{Sis-etal07},
enhancing ABC with nonlinear regression models~\citep{BluFra10},
agent-based modelling~\citep{SotTav10}, and a Gibbs sampling ABC
approach~\citep{Wil-etal11}.  A survey of recent developments is given
by~\citet{MarEtAl12}.

ABC methods allow for inference in a Bayesian setting where the
parameter $\theta\in\R^p$ of a statistical model is assumed to be
random with a given prior distribution~$f_\theta$, we have observed
data $X \in\R^d$ from a given distribution~$f_{X|\theta}$ depending on
$\theta$ and we want to use these data to draw inference
about~$\theta$.  In the areas where ABC methods are used, the
likelihood $f_{X|\theta}$ is typically not available in an explicit
form.  The term ``likelihood-free'' is used to indicate that ABC
methods do not make use of the likelihood $f_{X|\theta}$, but only
work with samples from the joint distribution of $\theta$ and $X$.  In
the context of ABC methods, the data are usually summarised using a
statistic $S\colon \R^d\to \R^q$, and analysis is then based on $S(X)$
instead of~$X$.  The choice of $S$ affects the accuracy of the ABC
estimates: the estimates can only be expected to converge to the true
posterior if $S$ is a sufficient statistic, otherwise additional error
is introduced.

The basic idea of ABC methods is to replace samples from the exact
posterior distribution $\theta\,|\,X\!=\!x^\ast$ or
$\theta\,|\,S(X)\!=\!s^\ast$ with samples from an approximating
distribution like $\theta \,|\, S(X)\!\approx\!  s^\ast$.  There are
many variants of the basic ABC method available, including different
implementations of the condition $S(X) \approx s^\ast$.  All variants
use a tolerance parameter~$\delta$ which controls the trade-off
between fast generation of samples (large values of $\delta$) and
accuracy of samples (small values of~$\delta$).  The easiest approach
to implement the condition $S(X) \approx s^\ast$, considered in this
paper, is to use $\| S(X) - s^\ast \| \leq \delta$ where $\|\quark\|$
is some norm on~$\R^q$.  Different approaches to choosing the
statistic~$S$ are used; a semi-automatic approach is described in
\citet{FeaPra12}.  In many cases considerable improvements can be
achieved by choosing the norm for comparison of $S(X)$ to $s^\ast$ in
a problem-specific way.

Despite the popularity of ABC methods, theoretical analysis is still
in its infancy.  The aim of this article is to provide a foundation
for such analysis by providing rigorous results about the convergence
properties of the ABC method.  Here, we restrict discussion to the
most basic variant, to set a baseline to which different ABC variants
can be compared.  We consider Monte Carlo estimates of posterior
expectations, using the ABC samples for the estimate.
Proposition~\ref{p:convergence} shows that such ABC estimates converge
to the true value under very weak conditions; once this is established
we investigate the rate of convergence in theorem~\ref{t:rate}.
Similar results, but in the context of estimating posterior densities
rather than posterior expectations can be found in \citet{Blum10} and
\citet{Biau13}.

The choice of norm and of the tolerance parameter~$\delta$ are major
challenges in the practical application of ABC methods, but not many
results are available in the literature.  A numerical study of the
trade-off between accuracy and computational cost, in the context of
sequential ABC methods, can be found in~\citet{SiFiStu13}.
\citet{Wil13} establishes that an ABC method which accepts or rejects
proposals with a probability based on the difference between the
observed and proposed data converges to the correct solution under
assumptions on model or measurement error.

The error of an ABC estimate is affected both by the bias of the ABC
samples, controlled by the tolerance parameter~$\delta$, and by Monte
Carlo error, controlled by the number~$n$ of accepted ABC samples.
One of the key ingredients in our analysis is the result, shown in
lemma~\ref{l:bias}, that for ABC estimates the bias satisfies
\begin{equation*}
  \mathrm{bias} \sim \delta^2.
\end{equation*}
Similarly,
in lemma~\ref{l:cost} we show that for
the ABC estimate we have
\begin{equation*}
  \mathrm{cost} \sim n \delta^{-q},
\end{equation*}
where $q$ is the dimension of the observation~$s^\ast$.

It is well-known that for Monte Carlo methods the error decays, as a
function of computational cost, proportional to $\cost^{-1/2}$, where
the exponent $-1/2$ is independent of dimension \citep[see
\textit{e.g.}][section~3.2.2]{Vo13}.  In contrast, the main result of
this article, theorem~\ref{t:rate}, shows that, under optimal choice
of $\delta$, the basic ABC method satisfies
\begin{equation*}
  \mathrm{error} \sim \mathrm{cost}^{-2 / (q+4)}.
\end{equation*}
Corollary~\ref{c:rate-hat} shows that this result holds whether we fix
the number of accepted samples (controlling the precision of our
estimates but allowing the computational cost to be random) or fix the
number of proposals (allowing the number of accepted samples to be
random).  The former is aesthetically more satisfying, but in practice
most users have a fixed computational budget and hence must fix the
number of proposals.  In either case, the rate of decay for the error
gets worse as the dimension~$q$ increases and even in the
one-dimensional case the exponent $-2/(1+4) = -2/5$ is worse than the
exponent $-1/2$ for Monte Carlo methods.  \citet{FeaPra12} obtain the
same exponent $-2/(q+4)$ for the specific summary statistic $S(X) =
\E\cond{\theta}{X}$.  For the problem of estimating the posterior
density, \citet{Blum10} reports the slightly worse exponent
$-2/(q+5)$.  The difference between Blum's results and ours is due to
the fact that for kernel density estimation an additional bandwidth
parameter must be considered.

We continue by giving a very short introduction to the basic ABC
method in section~\ref{S:ABC}.  The main results,
proposition~\ref{p:convergence} and theorem~\ref{t:rate} are presented
in section~\ref{S:results}, together with their proofs.
Section~\ref{S:fixed-N} shows that our results are independent of
whether the number of proposals or of accepted samples are fixed
in the algorithm.
Section~\ref{S:experiments} illustrates the results of this paper with
the help of numerical experiments.  Finally, in
section~\ref{S:discussion}, we consider the practical implications of
our results.


\section{Approximate Bayesian Computation}
\label{S:ABC}

This section gives a short introduction to the basic ABC algorithm.  A
more complete description is, for example, given
in~\citet[section~5.1]{Vo13}.  We describe the algorithm in the
context of the following Bayesian inference problem:
\begin{itemize}
\item A parameter vector $\theta\in\R^p$ is assumed to be random.
  Before observing any data, our belief about its value is summarised
  by the prior distribution~$f_\theta$.  The value of~$\theta$ is
  unknown to us and our aim is to make inference about~$\theta$.
\item The available data $X \in\R^d$ are assumed to be a sample from a
  distribution $f_{X|\theta}$, depending on the parameter~$\theta$.
  Inference about $\theta$ is based on a single observed sample
  $x^\ast$ from this distribution; repeated observations can be
  assembled into a single vector if needed.
\item In the context of ABC, the data are often summarised using a
  statistic $S\colon \R^d \to \R^q$.  Since $X$ is random, $S = S(X)$
  is random with a distribution $f_{S|\theta}$ depending on~$\theta$.
  If a summary statistic is used, inference is based on the value
  $s^\ast = S(x^\ast)$ instead of on the full sample~$x^\ast$.
\end{itemize}
Our aim is to explore the posterior distribution of $\theta$,
\textit{i.e.}\ the conditional distribution of $\theta$ given $X =
x^\ast$, using Monte Carlo methods.  More specifically, we aim to
estimate the posterior expectations
\begin{equation}\label{eq:ABC-answer}
  y = E\condl{h(\theta)}{X=x^\ast}
\end{equation}
for given test functions~$h\colon \R^p\to \R$.  Expectations of this
form allow us to study many relevant properties of the posterior
distribution, including the posterior mean when $h(\theta) = \theta_i$
with $i=1, \ldots, p$ and posterior second moments when $h(\theta) =
\theta_i \theta_j$ with $i,j=1, \ldots, p$.  If $h(\theta) =
1_A(\theta)$ for a given set $A\subseteq\R^p$, the expectation
$E\condl{h(\theta)}{X=x^\ast}$ equals the posterior probability of
hitting~$A$; for example, the CDF can be approximated by choosing
$A=(-\infty, a]$ for $a\in\R$.

The basic ABC method for generating approximate samples from the
posterior distribution is given in the following algorithm.

\begin{algorithm}\label{alg:ABC}
  For a given observation $s^\ast\in\R^q$ and given tolerance
  $\delta>0$, {\em ABC samples} approximating the distribution
  $\theta\,|\,S=s^\ast$ are random samples $\theta^{(\delta)}_j
  \in \R^p$ computed by the following algorithm:
  \begin{algorithmic}[1]
    \STATE let $j \gets 0$
    \WHILE{$j < n$}
    \STATE sample $\theta$ with density $f_\theta(\quark)$
    \STATE sample $X$ with density $f_{X|\theta}\cond{\quark}{\theta}$
    \IF{$\|S(X) - s^\ast\|_A \leq \delta$}
    \STATE let $j \gets j + 1$
    \STATE let $\theta^{(\delta)}_j \gets \theta$
    \ENDIF
    \ENDWHILE
  \end{algorithmic}
\end{algorithm}

\noindent
The norm $\| \quark \|_A$ used in the acceptance criterion is defined
by $\| s \|_A^2 = s^\top A^{-1} s$ for all $s\in\R^q$, where $A$ is a
positive definite symmetric matrix.  This includes the case of the
Euclidean distance $\| \quark \| = \| \quark \|_I$ for $A=I$.  In
practical applications, the matrix~$A$ is often chosen in a
problem-specific way.

Using the output of the algorithm, an estimate for the posterior
expectation~\eqref{eq:ABC-answer} can be obtained as
\begin{equation}\label{eq:ABC-MC-basic}
  \Yn = \frac{1}{n} \sum_{j=1}^n h(\theta^{(\delta)}_j)
\end{equation}
where $\theta^{(\delta)}_1, \ldots, \theta^{(\delta)}_n$ are computed
by algorithm~\ref{alg:ABC}.  Since the output of the ABC algorithm
approximates the posterior distribution, the Monte Carlo
estimate~$\Yn$ can be used as an approximation to the posterior
expectation.

The ABC samples are only approximately distributed according to the
posterior distribution, and thus the estimate $\Yn$ will not exactly
converge to the true value~$y$ as $n\to\infty$.  The quality of the
approximation can be improved by decreasing the tolerance
parameter~$\delta$, but this leads at the same time to a lower
acceptance probability in algorithm~\ref{alg:ABC} and thus,
ultimately, to higher computational cost for obtaining the estimate
$\Yn$.  Hence, a trade-off must be made between accuracy of the
results and speed of computation.

Since the algorithm is not using $x^\ast$ directly, but uses $s^\ast =
S(x^\ast)$ instead, we require $S$ to be a sufficient statistic, so
that we have
\begin{equation*}
  y
  = E\condl{h(\theta)}{X=x^\ast}
  = E\condl{h(\theta)}{S=s^\ast}.
\end{equation*}
If $S$ is not sufficient, an additional error will be introduced.

For {\em application} of the ABC method, knowledge of the
distributions of $\theta$, $X$ and $S=S(X)$ is not required; instead
we assume that we can simulate large numbers of samples of these
random variables.  In contrast, in our {\em analysis} we will assume
that the joint distribution of $\theta$ and $S$ has a density
$f_{S,\theta}$ and we will need to consider properties of this density
in some detail.

To conclude this section, we remark that there are two different
approaches to choosing the sample size used to compute the ABC
estimate~$\Yn$.  If we denote the number of proposals required to
generate $n$ output samples by $N \geq n$, then one approach is to
choose the number $N$ of proposals as fixed; in this case the number
$n \leq N$ of accepted proposals is random.  Alternatively, for given
$n$, the loop in the ABC algorithm could be executed until $n$ samples
are accepted, resulting in random~$N$ and fixed~$n$.  In order to
avoid complications with the definition of $\Yn$ for $n=0$, we follow
the second approach here and defer discussion of the case of fixed $N$
until section~\ref{S:fixed-N}.


\section{Results}
\label{S:results}

This section presents the main results of the paper, in
proposition~\ref{p:convergence} and theorem~\ref{t:rate}, followed by
proofs of these results.

Throughout, we assume that the joint distribution of $\theta$ and $S$
has a density $f_{S,\theta}$, and we consider the marginal densities
of $S$ and $\theta$ given by
\begin{equation*}
  f_S(s) = \int_{\R^p} f_{S,\theta}(s,t) \,dt
\end{equation*}
for all $s\in\R^q$ and
\begin{equation*}
  f_\theta(t) = \int_{\R^q} f_{S,\theta}(s,t) \,ds
\end{equation*}
for all $t\in \R^p$, respectively.  We also consider the conditional
density of $\theta$ given $S=s$, defined by
\begin{equation*}
  f_{\theta|S}\cond{t}{s}
  = \begin{cases}
  \frac{f_{S,\theta}(s,t)}{f_S(s)}, & \mbox{if $f_S(s) > 0$, and} \\
  0 & \mbox{otherwise.}
  \end{cases}
\end{equation*}

Our aim is to study the convergence of the estimate~$\Yn$ to $y$ as $n
\to \infty$.  The fact that ABC estimates converge to the correct
value as $\delta \downto 0$ is widely accepted and easily shown
empirically.  This result is made precise in the following
proposition, showing that the convergence holds under very weak
conditions.

\begin{prop}\label{p:convergence}
  Let $h\colon \R^p\to \R$ be such that $\E\bigl(|h(\theta)|\bigr) <
  \infty$.  Then, for $f_S$-almost all $s^\ast\in\R^q$, the ABC
  estimate $\Yn$ given by~\eqref{eq:ABC-MC-basic} satisfies
  \begin{enumerate}
  \item $\displaystyle \lim_{n\to\infty} \Yn = \E\bigl( \Yn \bigr)$
    almost surely for all $\delta>0$; and
  \item $\displaystyle \lim_{\delta\downto 0} \E\bigl( \Yn \bigr)
    = E\condl{h(\theta)}{S=s^\ast}$
    for all $n\in\N$.
  \end{enumerate}
\end{prop}

We note that the assumptions required in our version of this result
are very modest.  Since our aim is to estimate the posterior
expectation of $h(\theta)$, the assumption that the prior expectation
$\E\bigl(|h(\theta)|\bigr)$ exists is reasonable.  Similarly, the
phrase ``for $f_S$-almost all $s^\ast\in\R^q$'' in the proposition
indicates that the result holds for all $s^\ast$ in a
set~$A\subseteq\R^q$ with $P\bigl(S(X)\in A\bigr) = 1$.  Since in
practice the value $s^\ast$ will be an observed sample of $S(X)$, this
condition forms no restriction to the applicability of the result.
The proposition could be further improved by removing the assumption
that the distributions of $\theta$ and $S(X)$ have densities.  While a
proof of this stronger result could be given following similar lines
to the proof given below, here we prefer to keep the setup consistent
with what is required for the result in theorem~\ref{t:rate}.

\bigskip

Our second result, theorem~\ref{t:rate}, quantifies the speed of
convergence of $\Yn$ to~$y$.  We consider the mean squared error
\begin{equation*}
  \MSE(\Yn)
  = \E\bigl( ( \Yn - y )^2 \bigr)
  = \Var(\Yn) + \bias(\Yn)^2,
\end{equation*}
and relate this to the computational cost of computing the
estimate~$\Yn$.  For the computational cost, rather than using a
sophisticated model of computation, we restrict ourselves to the
na{\"\i}ve approach of assuming that the time required to obtain the
estimate $\Yn$ in algorithm~\ref{alg:ABC}, denoted $\cost(\Yn)$, is
proportional to $a + bN$, where $a$ and~$b$ are constants and $N$ is
the number of iterations the while-loop in the algorithm has to
perform until the condition $j = n$ is met (\textit{i.e.}\ the number
of proposals required to generate $n$ samples).  To describe the
asymptotic behaviour of MSE and cost, we use the following notation.

\begin{notation}
  For sequences $(a_n)_{n\in\N}$ and $(b_n)_{n\in\N}$ of positive real
  numbers we write $a_n \sim b_n$ to indicate that the limit $c =
  \lim_{n\to\infty} a_n / b_n$ exists and satisfies $0 < |c| < \infty$.
\end{notation}

Our result about the speed of convergence requires the density of
$S(X)$ to have continuous third partial derivatives.  More
specifically, we will use the following technical assumptions on~$S$.

\begin{assumption}
  The density $f_S$ and the function $s \mapsto \int_{\R^p} h(t)
  \,f_{S,\theta}(s,t) \,dt$ are three times continuously
  differentiable in a neighbourhood of~$s^\ast\in\R^q$.
\end{assumption}

\begin{theorem}\label{t:rate}
  Let $h\colon \R^p\to \R$ be such that $\E\bigl(h(\theta)^2\bigr) <
  \infty$ and let $S$ be sufficient and satisfy assumption~A.  Assume
  $\Var\condl{h(\theta)}{S=s^\ast} > 0$ and
  \begin{equation}\label{eq:C-s-ast}
    C(s^\ast)
    = \frac{\Delta \phi_h(s^\ast) - y \cdot \Delta\phi_1(s^\ast)}%
           {2(q+2) \phi_1(s^\ast)}
    \neq 0.
  \end{equation}
  Then, for $f_S$-almost
  all $s^\ast$, the following statements hold:
  \begin{enumerate}
  \item Let $(\delta_n)_{n\in\N}$ be a sequence with $\delta_n \sim
    n^{-1/4}$.  Then the mean squared error satisfies
    \begin{equation*}
      \MSE(\Ynn) \sim \E\bigl(\cost(\Ynn)\bigr)^{-4/(q+4)}
    \end{equation*}
    as $n\to\infty$.
  \item The exponent $-4/(q+4)$ given in the preceding statement is
    optimal: for any sequence $(\delta_n)_{n\in\N}$ with $\delta_n
    \downarrow 0$ as $n\to \infty$ we have
    \begin{equation*}
      \liminf_{n\to\infty}
        \frac{\!\MSE(\Ynn)\;}{\;\E\bigl(\cost(\Ynn)\bigr)^{-4/(q+4)}\!}
      > 0.
    \end{equation*}
  \end{enumerate}
\end{theorem}

The rate of convergence in the first part of theorem~\ref{t:rate}
should be compared to the corresponding rate for the usual Monte Carlo
estimate.  Since Monte Carlo estimates are unbiased, the root-mean
squared error for a Monte Carlo estimate is proportional to
$1/\sqrt{n}$, while the cost of generating $n$ samples is proportional
to $n$.  Thus, for Monte Carlo estimates we have $\RMSE \sim
\cost^{-1/2}$.  The corresponding exponent from theorem~\ref{t:rate},
obtained by taking square roots, is $-2/(q+4)$, showing slower
convergence for ABC estimates.  This reduced efficiency is a
consequence of the additional error introduced by the bias in the ABC
estimates.

The statement of the theorem implies that, as computational effort is
increased, $\delta$ should be decreased proportional to $n^{-1/4}$.
For this case, the error decreases proportionally to the cost with
exponent $r = -4/(q+4)$.  The second part of the theorem shows that no
choice of $\delta$ can lead to a better (\textit{i.e.}\ more negative)
exponent.  If the tolerances $\delta_n$ are decreased in a way which
leads to an exponent $\tilde r > r$, in the limit for large values of
$n$ such a schedule will always be inferior to the choice $\delta_n =
c n^{-1/4}$, for any constant~$c>0$.  It is important to note that
this result only applies to the limit $n\to\infty$.

The result of theorem~\ref{t:rate} describes the situation where, for
a fixed problem, the tolerance~$\delta$ is decreased.  Caution is
advisable when comparing ABC estimates for different~$q$.  Since the
constant implied in the $\sim$-notation depends on the choice of
problem, and thus on~$q$, the dependence of the error on $q$ for
fixed~$\delta$ is not in the scope of theorem~\ref{t:rate}.  Indeed,
\citet{Blum10} suggests that in practice ABC methods can be
successfully applied for larger values of~$q$, despite results like
theorem~\ref{t:rate}.

\bigskip

Before we turn our attention to proving the results stated above, we
first remark that, without loss of generality, we can assume that the
acceptance criterion in the algorithm uses Euclidean distance $\|
\quark \|$ instead of $\| \quark \|_A$.  This can be achieved by
considering the modified statistic $\tilde S(x) = A^{-1/2} S(x)$ and
$\tilde s^\ast = \tilde S(x^\ast) = A^{-1/2} s^\ast$, where $A^{-1/2}$
is the positive definite symmetric square root of~$A^{-1}$.  Since
\begin{align*}
  \| S(X) - s^\ast \|_A^2
  &= \bigl(S(X) - s^\ast\bigr)^\top A^{-1} \bigl(S(X) - s^\ast\bigr) \\
  &= \Bigl( A^{-1/2} \bigl(S(X) - s^\ast\bigr) \Bigr)^\top
     \Bigl( A^{-1/2} \bigl(S(X) - s^\ast\bigr) \Bigr) \\
  &= \| \tilde S(X) - \tilde s^\ast \|^2,
\end{align*}
the ABC algorithm using $S$ and $\|\quark\|_A$ has identical output to
the ABC algorithm using $\tilde S$ and the Euclidean norm.  Finally, a
simple change of variables shows that the assumptions of
proposition~\ref{p:convergence} and theorem~\ref{t:rate} are satisfied
for $S$ and $\|\quark\|_A$ if and only they are satisfied for $\tilde
S$ and $\|\quark\|$.  Thus, for the proofs we will assume that
Euclidean distance is used in the algorithm.

The rest of this section contains the proofs of
proposition~\ref{p:convergence} and theorem~\ref{t:rate}.
In the proofs it will be convenient to use the following technical
notation.

\begin{definition}\label{def:phi}
  For $h\colon \R^p\to \R$ with $\E\bigl(|h(\theta)|\bigr) < \infty$
  we define $\phi_h\colon \R^q\to \R$ to be
  \begin{equation*}
    \phi_h(s)
    = \int_{\R^p} h(t) f_{S,\theta}(s, t) \,dt
  \end{equation*}
  for all $s\in\R^q$ and $\phi\avg_h\colon \R^q\to \R$ to be
  \begin{equation}\label{eq:phi-delta}
    \phi\avg_h(s^\ast)
    = \frac{1}{\bigl|B(s^\ast, \delta)\bigr|}
            \int_{B(s^\ast, \delta)} \phi_h(s) \,ds,
  \end{equation}
  for all $s^\ast \in\R^q$, where $\bigl|B(s^\ast, \delta) \bigr|$
  denotes the volume of the ball $B(s^\ast, \delta)$.
\end{definition}

Using the definition for $h\equiv 1$ we get $\phi_1 \equiv f_S$ and
for general $h$ we get $\phi_h(s) = f_S(s) \,
\E\condl{h(\theta)}{S=s}$.  Both the exact value~$y$
from~\eqref{eq:ABC-answer} and the mean of the estimator $\Yn$
from~\eqref{eq:ABC-MC-basic} can be expressed in the notation of
definition~\ref{def:phi}.  This is shown in the following lemma.

\begin{lemma}\label{l:phi}
  Let $h\colon \R^p\to \R$ be such that $\E\bigl(|h(\theta)|\bigr) <
  \infty$.  Then
  \begin{equation*}
    \E\condl{h(\theta)}{S=s^\ast}
    = \frac{\phi_h(s^\ast)}{\phi_1(s^\ast)}
  \end{equation*}
  and
  \begin{equation*}
    \E\bigl( \Yn \bigr)
    = \E\bigl( h(\theta^{(\delta)}_1) \bigr)
    = \frac{\phi\avg_h(s^\ast)}{\phi\avg_1(s^\ast)}
  \end{equation*}
  for $f_S$-almost all $s^\ast\in\R^q$.
\end{lemma}

\begin{proof}
  From the assumption $\E\bigl(|h(\theta)|\bigr) < \infty$ we can
  conclude
  \begin{equation*}
    \begin{split}
    \int_{\R^q} \int_{\R^p} \bigl| h(t) \bigr| f_{S,\theta}(s,t) \,dt \,ds
    &= \int_{\R^p} \bigl| h(t) \bigr|
                \int_{\R^q} f_{S,\theta}(s,t) \,ds \,dt \\
    &= \int_{\R^p} \bigl| h(t) \bigr| f_\theta(t) \,dt
    = \E\bigl(|h(\theta)|\bigr)
    < \infty,
    \end{split}
  \end{equation*}
  and thus we know that $\int_{\R^p} \bigl| h(t) \bigr|
  f_{S,\theta}(s,t) \,dt < \infty$ for almost all $s\in\R^q$.
  Consequently, the conditional distribution
  $\E\condl{h(\theta)}{S=s^\ast}$ exists for $f_S$-almost all
  $s^\ast\in\R^q$.  Using Bayes' rule we get
  \begin{equation*}
    \begin{split}
    \E\condl{h(\theta)}{S=s^\ast}
    &= \int_{\R^p} h(t) f_{\theta|S}\cond{t}{s^\ast} \,dt \\
    &= \int_{\R^p} h(t) \frac{f_{S,\theta}(s^\ast, t)}{f_S(s^\ast)} \,dt
    = \frac{\phi_h(s^\ast)}{\phi_1(s^\ast)}.
    \end{split}
  \end{equation*}

  On the other hand, the samples $\theta^{(\delta)}_j$ are distributed
  according to the conditional distribution of $\theta$ given $S \in
  B(s^\ast, \delta)$.  Thus, the density of the samples
  $\theta^{(\delta)}_j$ can be written as
  \begin{equation*}
    f_{\theta^{(\delta)}_j}(t)
    = \frac{1}{Z}
            \frac{1}{\bigl|B(s^\ast, \delta)\bigr|}
            \int_{B(s^\ast, \delta)} f_{S,\theta}(s,t) \,ds,
  \end{equation*}
  where the normalising constant $Z$ satisfies
  \begin{equation*}
    Z
    = \frac{1}{\bigl|B(s^\ast, \delta)\bigr|}
            \int_{B(s^\ast, \delta)} f_S(s) \,ds
    = \frac{1}{\bigl|B(s^\ast, \delta)\bigr|}
            \int_{B(s^\ast, \delta)} \phi_1(s) \,ds
    = \phi\avg_1(s^\ast),
  \end{equation*}
  and we get
  \begin{equation*}
  \begin{split}
    \E\bigl( \Yn \bigr)
    &= \E\bigl( h(\theta^{(\delta)}_1) \bigr) \\
    &= \frac{1}{\phi\avg_1(s)} \int_{\R^p} h(t)
             \frac{1}{\bigl|B(s^\ast, \delta)\bigr|}
            \int_{B(s^\ast, \delta)} f_{S,\theta}(s,t) \,ds \,dt \\
    &=  \frac{1}{\phi\avg_1(s)}
           \frac{1}{\bigl|B(s^\ast, \delta)\bigr|}
           \int_{B(s^\ast, \delta)} \phi_h(s) \,ds \\
    &= \frac{\phi\avg_h(s^\ast)}{\phi\avg_1(s^\ast)}.
  \end{split}
  \end{equation*}
  This completes the proof.
\end{proof}

Using these preparations, we can now present a proof of
proposition~\ref{p:convergence}.

\begin{proof}[Proof of proposition~\ref{p:convergence}]
  Since $\E\bigl( |h(\theta)| \bigr) < \infty$ we have
  \begin{equation*}
    \E\bigl(|\Yn|\bigr)
    \leq \E\bigl(|h(\theta_1)|\bigr)
    = \frac{\phi_{|h|}(s^\ast)}{\phi_1(s^\ast)} <\infty
  \end{equation*}
  whenever $\phi_1(s^\ast) = f_S(s^\ast) > 0$, and by the law of large
  numbers $\Yn$ converges to $\E\bigl( \Yn \bigr)$ almost surely.

  For the second statement, since $\phi_1 \equiv f_S \in L^1(\R^q)$,
  we can use the Lebesgue differentiation theorem
  \citep[theorem~7.7]{Rudin66} to conclude that $\phi\avg_1(s^\ast)
  \to \phi_1(s^\ast)$ as $\delta \downto 0$ for almost all
  $s^\ast\in\R^q$.  Similarly, since
  \begin{equation*}
    \int_{\R^q} \bigl| \phi_h(s) \bigr| \,ds
    \leq \int_{\R^p} \bigl| h(t) \bigr| \int_{\R^q} f_{S,\theta}(s,t) \,ds \,dt
    = \int_{\R^p} \bigl| h(t) \bigr| f_\theta(t) \,dt
    < \infty
  \end{equation*}
  and thus $\phi_h \in L^1(\R^q)$, we have $\phi\avg_h(s^\ast) \to
  \phi_h(s^\ast)$ as $\delta \downto 0$ for almost all $s^\ast\in\R^q$
  and using lemma~\ref{l:phi} we get
  \begin{equation*}
    \lim_{\delta\downto 0} \E\bigl( \Yn \bigr)
    = \lim_{\delta\downto 0} \frac{\phi\avg_h(s^\ast)}{\phi\avg_1(s^\ast)}
    = \frac{\phi_h(s^\ast)}{\phi_1(s^\ast)}
    = \E\condl{h(\theta)}{S=s^\ast}
  \end{equation*}
  for almost all $s^\ast\in\R^q$.  This completes the proof.
\end{proof}

For later use we also state the following simple consequence of
proposition~\ref{p:convergence}.

\begin{corollary}\label{C:variance}
  Assume that $\E\bigl(h(\theta)^2\bigr) < \infty$.  Then, for
  $f_S$-almost all $s^\ast\in\R^q$ we have
  \begin{equation*}
    \lim_{\delta\downarrow0} n \Var(\Yn)
    = \Var\condl{h(\theta)}{S=s^\ast},
  \end{equation*}
  uniformly in~$n$.
\end{corollary}

\begin{proof}
  From the definition of the variance we know
  \begin{equation*}
    \Var\bigl( \Yn \bigr)
    = \frac{1}{n} \Var\bigl( h(\theta^{(\delta)}_j) \bigr)
    = \frac{1}{n} \Bigl( \E\bigl( h(\theta^{(\delta)}_j)^2 \bigr)
            - \E\bigl( h(\theta^{(\delta)}_j) \bigr)^2 \Bigr).
  \end{equation*}
  Applying proposition~\ref{p:convergence} to the function $h^2$
  first, we get $\lim_{\delta\downarrow0}
  \E\bigl(h(\theta^{(\delta)}_j)^2\bigr) =
  \E\condl{h(\theta)^2}{S=s^\ast}$.  Since $\E\bigl(h(\theta)^2\bigr)
  < \infty$ implies $\E\bigl(|h(\theta)|\bigr) < \infty$, we also get
  $\lim_{\delta\downarrow0} \E\bigl(h(\theta^{(\delta)}_j)\bigr) =
  \E\condl{h(\theta)}{S=s^\ast}$ and thus
  \begin{align*}
    \lim_{\delta\downarrow 0} n \Var\bigl( \Yn \bigr)
    &= \E\condl{h(\theta)^2}{S=s^\ast} - \E\condl{h(\theta)}{S=s^\ast}^2 \\
    &= \Var\condl{h(\theta)}{S=s^\ast}.
  \end{align*}
  This completes the proof.
\end{proof}

The rest of this section is devoted to a proof of
theorem~\ref{t:rate}.  We first consider the bias of the
estimator~$\Yn$.  As is the case for Monte Carlo estimates, the bias
of the ABC estimate~$\Yn$ does not depend on the sample size~$n$.  The
dependence on the tolerance parameter~$\delta$ is given in the
following lemma.  This lemma is the key ingredient in the proof of
theorem~\ref{t:rate}.

\begin{lemma}\label{l:bias}
  Assume that $\E\bigl(|h(\theta)|\bigr) < \infty$ and that $S$
  satisfies assumption~A.  Then, for $f_S$-almost all $s^\ast\in\R^q$,
  we have
  \begin{equation*}
    \bias(\Yn)
    = C(s^\ast) \, \delta^2 + O(\delta^3)
  \end{equation*}
  as $\delta \downto 0$ where the constant $C(s^\ast)$ is given by
  equation~\eqref{eq:C-s-ast} and $\Delta$ denotes the Laplace operator.
\end{lemma}

\begin{proof}
  Using lemma~\ref{l:phi} we can write the bias as
  \begin{equation}\label{eq:bias}
    \bias(\Yn)
    = \frac{\phi\avg_h(s^\ast)}{\phi\avg_1(s^\ast)}
    - \frac{\phi_h(s^\ast)}{\phi_1(s^\ast)}.
  \end{equation}
  To prove the lemma, we have to study the rate of convergence of the
  averages~$\phi\avg_h(s^\ast)$ to the centre value $\phi_h(s^\ast)$
  as $\delta\downto 0$.  Using Taylor's formula we find
  \begin{equation*}
  \begin{split}
    \phi_h(s)
    &= \phi_h(s^\ast)
      + \nabla \phi_h(s^\ast) (s - s^\ast)
      + \frac12 (s - s^\ast)^\top H_{\phi_h}(s^\ast) (s-s^\ast) \\
    &\hskip1cm
      + r_3(s - s^\ast)
  \end{split}
  \end{equation*}
  where $H_{\phi_h}$ denotes the Hessian of $\phi_h$ and the error
  term $r_3$ satisfies
  \begin{equation}\label{eq:error-bound}
    \bigl| r_3(v) \bigr|
    \leq \max_{|\alpha|=3} \sup_{s\in B(s^\ast,\delta)}
                \bigl| \d^\alpha_s \phi_h(s) \bigr|
        \cdot \sum_{|\beta| = 3} \frac{1}{\beta!} \bigl| v^\beta \bigr|
  \end{equation}
  for all $s \in B(s^\ast, \delta)$, and $\d^\alpha_s$ denotes the
  partial derivative corresponding to the multi-index~$\alpha$.
  Substituting the Taylor approximation into
  equation~\eqref{eq:phi-delta}, we find
  \begin{equation}\label{eq:f-delta-1}
  \begin{split}
    \phi\avg_h(s^\ast)
    &= \frac{1}{\bigl|B(s^\ast, \delta)\bigr|} \int_{B(s^\ast, \delta)} \Bigl(
      \phi_h(s^\ast)
      + \nabla \phi_h(s^\ast) (s - s^\ast) \Bigr. \\
    &\hskip2cm \Bigl.
      + \frac12 (s - s^\ast)^\top H_{\phi_h}(s^\ast) (s-s^\ast)
      + r_3(s - s^\ast) \Bigr)
    \,ds \\
    &= \phi_h(s^\ast) + 0 \\
    &\hskip1cm
    + \frac{1}{2\bigl|B(s^\ast, \delta)\bigr|} \int_{B(s^\ast, \delta)}
      (s - s^\ast)^\top H_{\phi_h}(s^\ast) (s-s^\ast)
    \,ds \\
    &\hskip1cm
    + \frac{1}{\bigl|B(s^\ast, \delta)\bigr|} \int_{B(s^\ast, \delta)}
      r_3(s - s^\ast)
    \,ds.
  \end{split}
  \end{equation}
  Since the Hessian $H_{\phi_h}(s^\ast)$ is symmetric and since the
  domain of integration is invariant under rotations we can choose a
  basis in which $H_{\phi_h}(s^\ast)$ is diagonal, such that the
  diagonal elements coincide with the eigenvalues $\lambda_1, \ldots,
  \lambda_q$.  Using this basis we can write the quadratic term
  in~\eqref{eq:f-delta-1} as
  \begin{align*}
    \lhs \frac{1}{2\bigl|B(s^\ast, \delta)\bigr|} \int_{B(s^\ast, \delta)}
      (s - s^\ast)^\top H_{\phi_h}(s^\ast) (s-s^\ast)
    \,ds \\
    &= \frac{1}{2\bigl|B(0, \delta)\bigr|} \int_{B(0, \delta)}
      \sum_{i=1}^q \lambda_i u_i^2
    \,du \\
    &= \frac{1}{2\bigl|B(0, \delta)\bigr|} \sum_{i=1}^q \lambda_i
      \frac{1}{q} \sum_{j=1}^q \int_{B(0, \delta)} u_j^2 \,du \\
    &= \frac{1}{2\bigl|B(0, \delta)\bigr|} \sum_{i=1}^q \lambda_i
      \frac{1}{q} \int_{B(0, \delta)} |u|^2 \,du \\
    &= \tr H_{\phi_h}(s^\ast)
        \cdot \frac{1}{2q\bigl|B(0, \delta)\bigr|}
                    \int_{B(0, \delta)} |u|^2 \,du
  \end{align*}
  where $\tr H_{\phi_h} = \Delta \phi_h$ is the trace of the Hessian.
  Here we used the fact that the value $\int_{B(0, \delta)} u_i^2
  \,du$ does not depend on $i$ and thus equals the average
  $\frac{1}{q} \sum_{j=1}^q \int_{B(0, \delta)} u_j^2 \,du$.
  Rescaling space by a factor $1/\delta$ and using the relation
  $\int_{B(0,1)} |x|^2 \,dx = \bigl|B(0,1)\bigr| \cdot q/(q+2)$ we
  find
  \begin{align*}
    \lhs \frac{1}{2\bigl|B(s^\ast, \delta)\bigr|} \int_{B(s^\ast, \delta)}
      (s - s^\ast)^\top H_{\phi_h}(s^\ast) (s-s^\ast)
    \,ds \\
    &= \Delta\phi_h(s^\ast)
        \cdot \frac{1}{2q \delta^q \bigl|B(0, 1)\bigr|}
                    \int_{B(0, 1)} |y|^2 \delta^{q+2} \,dy \\
    &= \frac{\Delta\phi_h(s^\ast)}{2(q+2)} \delta^2.
  \end{align*}
  For the error term we can use a similar scaling argument in the
  bound~\eqref{eq:error-bound} to get
  \begin{equation*}
     \frac{1}{\bigl|B(s^\ast, \delta)\bigr|}
        \int_{B(s^\ast, \delta)} \bigl| r_3(s - s^\ast) \bigr| \,ds
     \leq C \cdot \delta^3
  \end{equation*}
  for some constant~$C$.  Substituting these
  results back into equation~\eqref{eq:f-delta-1} we find
  \begin{equation}\label{eq:f-delta-2}
    \phi\avg_h(s^\ast)
    = \phi_h(s^\ast) + a_h(s^\ast) \cdot \delta^2 + \CO(\delta^3)
  \end{equation}
  where
  \begin{equation*}
    a_h(s^\ast) = \frac{\Delta\phi_h(s^\ast)}{2(q+2)}.
  \end{equation*}
  Using formula~\eqref{eq:f-delta-2} for $h\equiv 1$ we also get
  \begin{equation*}
    \phi\avg_1(s^\ast)
    = \phi_1(s^\ast)
    + a_1(s^\ast) \cdot \delta^2
    + \CO(\delta^3)
  \end{equation*}
  as $\delta\downto 0$.

  Using representation~\eqref{eq:bias} of the bias, and omitting the
  argument $s^\ast$ for brevity, we can now express the bias in powers
  of~$\delta$:
  \begin{align*}
    \bias(\Yn)
    &= \frac{\phi_1\phi\avg_h - \phi\avg_1\phi_h}{\phi\avg_1\phi_1} \\
    &= \frac{\phi_1 \Bigl(
        \phi_h + a_h \delta^2 + \CO(\delta^3)
      \Bigr) - \Bigl(
        \phi_1 + a_1 \delta^2 + \CO(\delta^3)
      \Bigr)\phi_h}{\phi\avg_1\phi_1} \\
    &= \frac{\phi_1 a_h - \phi_h a_1}{\phi\avg_1\phi_1} \cdot  \delta^2
           + \CO(\delta^3).
  \end{align*}
  Since $1/\phi\avg_1 = 1/\phi_1 \cdot \bigl(1 + \CO(\delta^2)\bigr)$
  and $y = \phi_h / \phi_1$, the right-hand side can be simplified to
  \begin{equation*}
    \bias(\Yn)
    = \frac{a_h(s^\ast) - y a_1(s^\ast)}{\phi_1(s^\ast)} \cdot  \delta^2
           + \CO(\delta^3).
  \end{equation*}
  This completes the proof.
\end{proof}

To prove the statement of theorem~\ref{t:rate}, the bias of $\Yn$ has
to be balanced with the computational cost for computing $\Yn$.  When
$\delta$ decreases, fewer samples will satisfy the acceptance
condition $\| S(X) - s^\ast \| \leq \delta$ and the running time of
the algorithm will increase.  The following lemma makes this statement
precise.

\begin{lemma}\label{l:cost}
  Let $f_S$ be continuous at~$s^\ast$.  Then the expected
  computational cost for computing the estimate $\Yn$ satisfies
  \begin{equation*}
    \E\bigl(\cost(\Yn)\bigr)
    = c_1 + c_2 n \delta^{-q} \bigl( 1 + c_3(\delta) \bigr)
  \end{equation*}
  for all $n\in\N$, $\delta > 0$ where $c_1$ and $c_2$ are constants,
  and $c_3$ does not depend on~$n$ and satisfies $c_3(\delta) \to 0$
  as $\delta \downto 0$.
\end{lemma}

\begin{proof}
  The computational cost for algorithm~\ref{alg:ABC} is of the form $a
  + bN$ where $a$ and $b$ are constants, and $N$ is the random number
  of iterations of the loop required until $n$ samples are accepted.
  The number of iterations required to generate one ABC sample is
  geometrically distributed with parameter
  \begin{equation*}
    p = P\bigl(S \in B(s^\ast, \delta)\bigr)
  \end{equation*}
  and thus $N$, being the sum of $n$ independent geometrically
  distributed values, has mean $\E(N) = n / p$.  The probability~$p$
  can be written as
  \begin{equation*}
    p
    = \int_{B(s^\ast, \delta)} f_S(s) \,ds
    = \bigl| B(s^\ast, \delta) \bigr| \cdot \phi_1\avg(s^\ast)
    = \delta^q \bigl| B(s^\ast, 1) \bigr| \cdot \phi_1\avg(s^\ast).
  \end{equation*}
  Since $\phi_1 = f_S$ is continuous at~$s^\ast$, we have
  $\phi\avg_1(s^\ast) \to \phi_1(s^\ast)$ as $\delta \downto 0$ and
  thus $p = c \delta^q \bigl( 1 + o(1) \bigr)$ for some constant~$c$.
  Thus, we find that the computational cost satisfies
  \begin{align*}
    \E\bigl(\cost(\Yn)\bigr)
    &= a + b \cdot \E(N) \\
    &= a + b \cdot \frac{n}{c \delta^q \bigl( 1 + o(1) \bigr)} \\
    &= a + \frac{b}{c} \, n \delta^{-q} \, \frac{1}{1 + o(1)},
  \end{align*}
  and the proof is complete.
\end{proof}

Finally, the following two lemmata give the relation between the
approximation error caused by the bias on the one hand and the
expected computational cost on the other hand.

\begin{lemma}\label{l:optimal}
  Let $\delta_n \sim n^{-1/4}$ for $n\in\N$ .  Assume that $\E\bigl(
  h(\theta)^2 \bigr) < \infty$, that $S$ satisfies assumption~A
  and $ \Var\condl{h(\theta)}{S = s^\ast} > 0$.
  Then, for $f_S$-almost all $s^\ast\in\R^q$, the error satisfies
  $\MSE(\Ynn) \sim n^{-1}$, the expected computational cost increases
  as $\E\bigl(\cost(\Ynn)\bigr) \sim n^{1+q/4}$ and we have
  \begin{equation*}
    \MSE(\Ynn) \sim \E\bigl(\cost(\Ynn)\bigr)^{-4/(q+4)}
  \end{equation*}
  as $n\to\infty$.
\end{lemma}

\begin{proof}
  By assumption, the limit $D = \lim \delta_n n^{1/4}$ exists.  Using
  lemma~\ref{l:bias} and corollary~\ref{C:variance}, we find
  \begin{align*}
    \lim_{n\to\infty} \frac{\MSE(\Ynn)}{n^{-1}}
    &= \lim_{n\to\infty} n\Bigl( \Var(\Ynn) + \bigl( \bias(\Ynn) \bigr)^2 \Bigr) \\
    &= \lim_{n\to\infty} n \Var(\Ynn) + \lim_{n\to\infty}
                n \bigl( C(s^\ast) \, \delta_n^2 + O(\delta_n^3) \bigr)^2 \\
    &= \Var\condl{h(\theta)}{S = s^\ast} + \lim_{n\to\infty}
                \Bigl( C(s^\ast) + \frac{\CO(\delta_n^3)}{\delta_n^2} \Bigr)^2
                (\delta_n n^{1/4})^4 \\
    &=  \Var\condl{h(\theta)}{S = s^\ast} + C(s^\ast)^2 D^4.
  \end{align*}
  On the other hand, using lemma~\ref{l:cost}, we find
  \begin{align*}
    \lim_{n\to\infty} \frac{\E\bigl(\cost(\Ynn)\bigr)}{n^{1+q/4}}
    &= \lim_{n\to\infty}
        \frac{c_1 + c_2 n \delta_n^{-q} \bigl( 1 + c_3(\delta_n) \bigr)}{n^{1+q/4}} \\
    &= 0 + c_2 \lim_{n\to\infty} \frac{1}{(\delta _n n^{1/4})^q}
        \bigl(1 + \lim_{n\to\infty} c_3(\delta_n) \bigr) \\
    &= c_2 / D^q.
  \end{align*}
  Finally, combining the rates for cost and error we get the result
  \begin{equation}\label{eq:optimal}
  \begin{split}
    \lhs \lim_{n\to\infty} \frac{\!\MSE(\Ynn)\;}{\; \E\bigl(\cost(\Ynn)\bigr)^{-4/(q+4)} \!} \\
    &= \lim_{n\to\infty} \frac{\MSE(\Ynn)}{n^{-1}}
      / \lim_{n\to\infty} \Bigl( \frac{\E\bigl(\cost(\Ynn)\bigr)}{n^{1+q/4}} \Bigr)^{-4/(q+4)} \\
    &= \Bigl(  \Var\condl{h(\theta)}{S = s^\ast} + C(s^\ast)^2 D^4 \Bigr)
      \cdot \Bigl( \frac{D^q}{c_2} \Bigr)^{-4/(q+4)}.
  \end{split}
  \end{equation}
  Since the right-hand side of this equation is non-zero, this
  completes the proof.
\end{proof}

Lemma~\ref{l:optimal} only specifies the optimal {\em rate} for the
decay of $\delta_n$ as a function of~$n$.  By inspecting the proof, we
can derive the corresponding optimal constant: $\delta_n$ should be
chosen as $\delta_n = D n^{-1/4}$ where $D$ minimises the expression
in equation~\eqref{eq:optimal}.  The optimal value of $D$ can
analytically be found to be
\begin{equation*}
  D_{\mathrm{opt}}
  = \Bigl( \frac{q \Var\condl{h(\theta)}
                {S = s^\ast}}{4 C(s^\ast)^2} \Bigr)^{1/4},
\end{equation*}
where $C(s^\ast)$ is the constant from lemma~\ref{l:bias}.  The
variance $\Var\condl{h(\theta)}{S = s^\ast}$ is easily estimated, but
it seems difficult to estimate $C(s^\ast)$ in a likelihood-free way.

The following result completes the proof of theorem~\ref{t:rate} by
showing that no other choice of $\delta_n$ can lead to a better rate
for the error while retaining the same cost.

\begin{lemma}\label{l:bound}
  For $n\in\N$ let $\delta_n \downto 0$.  Assume that $\E\bigl(
  h(\theta)^2 \bigr) < \infty$ with $\Var\condl{h(\theta)}{S = s^\ast} > 0$,
  that $S$ satisfies assumption~A and
  $C(s^\ast) \neq 0$.
  Then, for $f_S$-almost all $s^\ast\in\R^q$, we have
  \begin{equation*}
    \liminf_{n\to\infty}\frac{\!\MSE(\Ynn)\;}{\;\E\bigl(\cost(\Ynn)\bigr)^{-4/(q+4)}\!}
    > 0.
  \end{equation*}
\end{lemma}

\begin{proof}
  From lemma~\ref{l:bias} we know that
  \begin{equation}\label{eq:bound1}
  \begin{split}
    \MSE(\Ynn)
    &= \Var\bigl(\Ynn\bigr) + \bigl( \bias(\Ynn) \bigr)^2 \\
    &= \frac{\Var\bigl(h(\theta^{(\delta_n)}_j)\bigr)}{n} + \bigl( C(s^\ast) \, \delta_n^2 + \CO(\delta_n^3) \bigr)^2 \\
    &= \frac{\Var\bigl(h(\theta^{(\delta_n)}_j)\bigr)}{n} + C(s^\ast)^2 \, \delta_n^4 + \CO(\delta_n^5) \\
    &\geq \frac{\Var\bigl(h(\theta^{(\delta_n)}_j)\bigr)}{n} + \frac{C(s^\ast)^2}{2} \, \delta_n^4
  \end{split}
  \end{equation}
  for all sufficiently large~$n$, since $C(s^\ast)^2 > 0$.
  By lemma~\ref{l:cost}, as $n\to\infty$,
  we have
  \begin{equation}\label{eq:bound2}
    \begin{split}
      \lhs \E\bigl( \cost(\Ynn) \bigr)^{-4/(q+4)} \\
      &\sim (n \delta_n^{-q})^{-4/(q+4)} \\
    &= \bigl(n^{-1}\bigr)^{4/(q+4)} \bigl(\delta_n^4\bigr)^{q/(q+4)} \\
    &\sim \Bigl(\frac{4}{q+4}
        \frac{\Var\bigl(h(\theta^{(\delta_n)}_j)\bigr)}{n} \Bigr)^{4/(q+4)}
      \Bigl(\frac{q}{q+4} \, \frac{C(s^\ast)^2}{2} \delta_n^4 \Bigr)^{q/(q+4)}
    \end{split}
  \end{equation}
  where we were able to insert the constant factors into the last term
  because the $\sim$-relation does not see constants.  Using Young's
  inequality we find
  \begin{equation}\label{eq:bound3}
  \begin{split}
    \lhs \frac{\Var\bigl(h(\theta^{(\delta_n)}_j)\bigr)}{n}
            + \frac{C(s^\ast)^2}{2} \delta_n^4 \\
    &\geq \Bigl(\frac{4}{q+4}
        \frac{\Var\bigl(h(\theta^{(\delta_n)}_j)\bigr)}{n} \Bigr)^{4/(q+4)}
      \Bigl(\frac{q}{q+4} \, \frac{C(s^\ast)^2}{2} \delta_n^4 \Bigr)^{q/(q+4)}
  \end{split}
  \end{equation}
  and combining equations \eqref{eq:bound1}, \eqref{eq:bound2}
  and~\eqref{eq:bound3} we get
  \begin{equation*}
    \MSE(\Ynn)
    \geq c \cdot \E\bigl( \cost(\Ynn) \bigr)^{-4/(q+4)}
  \end{equation*}
  for some constant $c > 0$ and all sufficiently large~$n$.  This is
  the required result.
\end{proof}

The statement of theorem~\ref{t:rate} coincides with the statements of
lemmata \ref{l:optimal} and~\ref{l:bound} and thus we have completed
the proof of theorem~\ref{t:rate}.


\section{Fixed Number of Proposals}
\label{S:fixed-N}

So far, we have fixed the number $n$ of accepted samples; the
number~$N$ of proposals required to generate $n$ samples is then a
random variable.  In this section we consider the opposite approach,
where the number $\hat N$ of proposals is fixed and the number $\hat
n$ of accepted samples is random.  Compared to the case with fixed
number of samples, the case considered here includes two additional
sources of error. First, with small probability no samples at all will
be accepted and we shall see that this introduces a small additional
bias.  Secondly, the randomness in the value of $\hat n$ introduces
additional variance in the estimate.  We show that the additional
error introduced is small enough not to affect the rate of
convergence.  The main result here is corollary~\ref{c:rate-hat},
below, which shows that the results from theorem~\ref{t:rate} still
hold for random~$\hat n$.

In analogy to the definition of $\Yn$ from
equation~\eqref{eq:ABC-MC-basic}, we define
\begin{equation*}
  \hYN = \begin{cases}
    \frac{1}{\hat n} \sum_{j=1}^{\hat n} h\bigl(\theta_j^{(\delta)}\bigr),
           & \mbox{if $\hat n > 0$, and} \\
    \hat c & \mbox{if $\hat n = 0$,}
  \end{cases}
\end{equation*}
where $\hat c$ is an arbitrary constant (\textit{e.g.}\ zero or the
prior mean) to be used if none of the proposals are accepted.

For a given tolerance $\delta > 0$, each of the $\hat N$ proposals is
accepted with probability
\begin{equation*}
  p = P\bigl(S(X) \in B(s^\ast, \delta)\bigr),
\end{equation*}
and the number of accepted samples is binomially distributed
with mean $\E(\hat n) = \hat N p$.  In this section we consider
arbitrary sequences of $\hat N$ and $\delta$ such that $\hat N p \to
\infty$ and we show that asymptotically the estimators $\hYn$ (using
$n = \lfloor \hat N p \rfloor$ accepted samples) and $\hYN$ (using
$\hat N$ proposals) have the same error.

\begin{prop}\label{p:MSE-fixed-N}
  Let $\hYN$ and $\hYn$ be as above.  Then
  \begin{equation*}
    \frac{\MSE\bigl(\hYN\bigr)}{\MSE\bigl(\hYn\bigr)}
    \longrightarrow 1
  \end{equation*}
  as $\hat N p \to \infty$.
\end{prop}

\begin{proof}
  We start the proof by comparing the variances of the two estimates
  $\hYN$ and $\hYn$.  For the estimate with a fixed number of
  proposals we find
  \begin{align*}
    \Var\bigl(\hYN\bigr)
    &= \E\Bigl(\Var\condl{\hYN}{\hat n}\Bigr)
            + \Var\Bigl(\E\condl{\hYN}{\hat n}\Bigr) \\
    &= \E\Bigl(\frac{\sigma_\delta^2}{\hat n} 1_{\{\hat n > 0\}}\Bigr)
    + (\hat c - y_\delta)^2 P(\hat n > 0) P(\hat n = 0),
  \end{align*}
  where $\sigma_\delta^2 = \Var\condl{h(\theta)}{S\in
    B(s^\ast,\delta)}$ and $y_\delta = \E\condl{h(\theta)}{S\in
    B(s^\ast,\delta)}$.  Thus,
  \begin{equation}\label{eq:Nhat-Var}
    \begin{split}
      \frac{\Var\bigl(\hYN\bigr)}{\Var\bigl(\hYn\bigr)}
      &= \frac{\E\Bigl(\frac{\sigma_\delta^2}{\hat n} 1_{\{\hat n > 0\}}\Bigr)
                 + (\hat c - y_\delta)^2 P(\hat n > 0) P(\hat n = 0)}%
              {\frac{\sigma_\delta^2}{\lfloor\hat N p\rfloor}} \\
      &\leq \E\Bigl(\frac{\hat N p}{\hat n} 1_{\{\hat n > 0\}}\Bigr)
        + \frac{(\hat c - y_\delta)^2}{\sigma_\delta^2}
                    \lfloor\hat N p\rfloor (1-p)^{\hat N}.
    \end{split}
  \end{equation}
  Our aim is to show that the right-hand side of this equation
  converges to~$1$.  To see this, we split the right-hand side into
  four different terms.  First, since $\log\bigl( (1-p)^{\hat N}\bigr)
  = \hat N \log(1-p) \leq -\hat N p$, we have $(1-p)^{\hat N} \leq
  \exp(-\hat N p)$ and thus
  \begin{equation}\label{eq:Nhat-term0}
    \lfloor\hat N p\rfloor (1-p)^{\hat N} \to 0,
  \end{equation}
  \textit{i.e.}\ the right-most term in \eqref{eq:Nhat-Var} disappears
  as $\hat N p \to\infty$.  Now let $\eps > 0$.  Then we have
  \begin{equation}\label{eq:Nhat-term1}
    \E\Bigl(\frac{\hat N p}{\hat n} 1_{\{\hat n > (1+\eps)\hat N p\}}\Bigr)
    \leq \frac{1}{1+\eps} \cdot P\Bigl(\frac{\hat n}{\hat N p} > 1+\eps\Bigr) \\
    \leq \frac{1}{\eps^2} \Var\Bigl(\frac{\hat n}{\hat N p}\Bigr)
    \to 0
  \end{equation}
  as $\hat N p\to\infty$ by Chebyshev's inequality.  For the lower
  tails of $\hat n / \hat N p$ we find
  \begin{align*}
    \E\Bigl(\frac{\hat N p}{\hat n}
                1_{\{0 < \hat n \leq (1-\eps)\hat N p\}}\Bigr)
    &\leq \hat N p P\Bigl( 0 < \hat n \leq (1-\eps)\hat N p \Bigr) \\
    &\leq (1-\eps) (\hat N p)^2
            P\Bigl( \hat n = \bigl\lfloor (1-\eps)\hat N p \bigr\rfloor \Bigr).
  \end{align*}
  Using the relative entropy
  \begin{equation*}
    H\cond{q}{p}
    = q \log\Bigl(\frac{q}{p}\Bigr) + (1-q)\log\Bigl(\frac{1-q}{1-p}\Bigr),
  \end{equation*}
  lemma~2.1.9 in~\citet{DeZei98} states $P(\hat n = k) \leq
  \exp\bigl(-\hat N H\cond{k/\hat N}{p}\bigr)$ and since $q \mapsto
  H\cond{q}{p}$ is decreasing for $q < p$ we get
  \begin{equation*}
    \E\Bigl(\frac{\hat N p}{\hat n} 1_{\{0 < \hat n \leq (1-\eps)\hat N p\}}\Bigr)
    \leq (1-\eps) (\hat N p)^2 \exp\Bigl(-\hat N \cdot H\condl{(1-\eps)p}{p}\Bigr).
  \end{equation*}
  It is easy to show that $H\condl{(1-\eps)p}{p} / p > \eps +
  (1-\eps)\log(1-\eps) > 0$ and thus
  \begin{equation}
    \begin{split}
      \lhs \E\Bigl(\frac{\hat N p}{\hat n}
                          1_{\{0 < \hat n \leq (1-\eps)\hat N p\}}\Bigr) \\
      &\leq (1-\eps) (\hat N p)^2 \exp\Bigl(-\hat N p
                          \bigl( \eps + (1-\eps)\log(1-\eps) \bigr) \Bigr)
      \longrightarrow 0
    \end{split}
  \end{equation}
  as $\hat Np \to \infty$.  Finally, we find
  \begin{equation}\label{eq:Nhat-term3a}
    \E\Bigl(\frac{\hat N p}{\hat n}
            1_{\{(1-\eps)\hat N p < \hat n \leq (1+\eps)\hat N p\}}\Bigr)
    \leq \frac{1}{1-\eps}
            P\Bigl( \bigl| \frac{\hat n}{\hat Np} - 1 \bigr| \leq \eps \Bigr)
    \longrightarrow \frac{1}{1-\eps}
  \end{equation}
  and
  \begin{equation}\label{eq:Nhat-term3b}
    \E\Bigl(\frac{\hat N p}{\hat n}
              1_{\{(1-\eps)\hat N p \leq \hat n \leq (1+\eps)\hat N p\}}\Bigr)
    \geq \frac{1}{1+\eps}
              P\Bigl( \bigl| \frac{\hat n}{\hat Np} - 1 \bigr| \leq \eps \Bigr)
    \longrightarrow \frac{1}{1+\eps}
  \end{equation}
  as $\hat N p\to \infty$.  Substituting equations
  \eqref{eq:Nhat-term0} to~\eqref{eq:Nhat-term3b} into
  \eqref{eq:Nhat-Var} and letting $\eps \downto0$ we finally get
  \begin{equation*}
    \frac{\Var\bigl(\hYN\bigr)}{\Var\bigl(\hYn\bigr)}
    \longrightarrow 1
  \end{equation*}
  as $\hat N p\to \infty$.

  From the previous sections we know $\bias(\Yn) = y_\delta - y$.
  Since
  \begin{equation*}
    \bias(\hYN)
    = \E\bigl( \E\cond{\hYN}{\hat n} - y \bigr)
    = (y_\delta - y) P(\hat n > 0) + (\hat c - y) P(\hat n = 0),
  \end{equation*}
  and since we already have seen $P(\hat n = 0) = (1-p)^{\hat N} =
  o(1/\hat Np)$, we find
  \begin{equation}\label{eq:Nhat-MSE1}
    \begin{split}
      \frac{\MSE(\hYN)}{\MSE(\hYn)}
      &= \frac{\Var(\hYN) + \bias(\hYN)^2}{\Var(\hYn) + \bias(\hYn)^2} \\
      &= \frac{\frac{\sigma_\delta^2}{\lfloor \hat N p\rfloor}\bigl(1+o(1)\bigr)
        + (y_\delta - y)^2 \bigl(1 + o(1)\bigr) + o( 1 / \hat N p)}
      {\frac{\sigma_\delta^2}{\lfloor \hat N p\rfloor} + (y_\delta - y)^2} \\
      &= \frac{\sigma_\delta^2\bigl(1+o(1)\bigr)
        + \lfloor \hat N p\rfloor(y_\delta - y)^2 \bigl(1 + o(1)\bigr) + o(1)}
      {\sigma_\delta^2 + \lfloor \hat N p\rfloor(y_\delta - y)^2}.
    \end{split}
  \end{equation}
  Let $\eps > 0$ and let $\hat N p$ be large enough that all $o(1)$
  terms in~\eqref{eq:Nhat-MSE1} satisfy $-\eps < o(1) < \eps$.  Then
  we have
  \begin{equation*}
    \frac{\MSE(\hYN)}{\MSE(\hYn)}
    \leq \frac{\sigma_\delta^2\bigl(1+\eps\bigr)
      + \lfloor \hat N p\rfloor(y_\delta - y)^2 \bigl(1 + \eps\bigr) + \eps}
    {\sigma_\delta^2 + \lfloor \hat N p\rfloor(y_\delta - y)^2}
    \leq \frac{\sigma_\delta^2(1+\eps) + \eps}{\sigma_\delta^2},
  \end{equation*}
  where the second inequality is found by maximising the function $x
  \mapsto \bigl(\sigma_\delta^2(1+\eps) + x(1+\eps) +
  \eps\bigr)/(x+\sigma_\delta^2)$.  Similarly, we find
  \begin{equation*}
    \frac{\MSE(\hYN)}{\MSE(\hYn)}
    \geq \frac{\sigma_\delta^2\bigl(1-\eps\bigr)
      + \lfloor \hat N p\rfloor(y_\delta - y)^2 \bigl(1 - \eps\bigr) + \eps}
    {\sigma_\delta^2 + \lfloor \hat N p\rfloor(y_\delta - y)^2}
    \geq \frac{\sigma_\delta^2(1-\eps) - \eps}{\sigma_\delta^2}
  \end{equation*}
  for all sufficiently large $\hat N p$.  Using these bounds, taking
  the limit $\hat N p \to \infty$ and finally taking the limit $\eps
  \downto 0$ we find
  \begin{equation*}
    \frac{\MSE(\hYN)}{\MSE(\hYn)}
    \longrightarrow 1
  \end{equation*}
  as $\hat N p \to \infty$.  This completes the proof.
\end{proof}

We note that the result of proposition~\ref{p:MSE-fixed-N} does not
require $p$ to converge to~$0$.  In this case neither
$\MSE\bigl(\hYN\bigr)$ nor $\MSE\bigl(\hYn\bigr)$ converges to~$0$,
but the proposition still shows that the difference between fixed $N$
and fixed $n$ is asyptotically negligible.

\begin{corollary}\label{c:rate-hat}
  Let the assumptions of theorem~\ref{t:rate} hold and let
  $(\delta_{\hat N})_{\hat N\in\N}$ be a sequence with $\delta_{\hat
    N} \sim \hat N^{-1/(q+4)}$, where $\hat N$ denotes the (fixed)
  number of proposals.  Then the mean squared error satisfies
  \begin{equation*}
    \MSE\bigl(\hat Y_{\hat N}^{\delta_{\hat N}}\bigr)
    \sim \cost\bigl(\hat Y_{\hat N}^{\delta_{\hat N}}\bigr)^{-4/(q+4)}
  \end{equation*}
  as $\hat N\to\infty$, and the given choice of $\delta_{\hat N}$ is
  optimal in the sense of theorem~\ref{t:rate}.
\end{corollary}

\begin{proof}
  Let $p_{\hat N} = P\bigl(S(X) \in B(s^\ast, \delta_{\hat N})\bigr)$.
  Then
  \begin{equation*}
    \hat N p_{\hat N}
    \sim \hat N \delta_{\hat N}^q
    \sim \hat N ^ {1 - q/(q+4)}
    = \hat N ^ {4/(q+4)}
    \to \infty
  \end{equation*}
  and thus we can apply proposition~\ref{p:MSE-fixed-N}.  Similarly,
  we have
  \begin{equation*}
    \delta_{\hat N}
    \sim \hat N ^ {- 1/(q+4)}
    \sim (\hat N p_{\hat N})^{-1/4}
    \sim \lfloor\hat N p_{\hat N}\rfloor^{-1/4}
  \end{equation*}
  and thus we can apply theorem~\ref{t:rate}.  Using these two results
  we get
  \begin{equation*}
    \MSE\bigl( \hat Y_{\hat N}^{\delta_{\hat N}} \bigr)
    \sim \MSE\bigl( Y^{(\delta_{\hat N})}_{\lfloor\hat N p_{\hat N}\rfloor} \bigr)
    \sim \E\Bigl(\cost\bigl( Y^{(\delta_{\hat N})}_{\lfloor\hat N
                  p_{\hat N}\rfloor} \bigr)\Bigr)^{-4/(q+4)}.
  \end{equation*}
  Since $\hat Y_{\hat N}^{\delta_{\hat N}}$ is always computed using
  $\hat N$ proposals we have
  \begin{equation*}
    \cost\Bigl(\hat Y_{\hat N}^{\delta_{\hat N}}\Bigr)
    \sim \hat N
  \end{equation*}
  and thus, using lemma~\ref{l:cost}, we find
  \begin{equation*}
    \E\Bigl(\cost\bigl( Y^{(\delta_{\hat N})}_{\lfloor\hat N p_{\hat N}\rfloor} \bigr)\Bigr)
    \sim \bigl\lfloor\hat N p_{\hat N}\bigr\rfloor \delta_{\hat N}^{-q}
    \sim \hat N
    \sim \cost\Bigl(\hat Y_{\hat N}^{\delta_{\hat N}}\Bigr).
  \end{equation*}
  This completes the proof of the first statement of the corollary.
  For the second statement there are two cases.  If $\delta_{\hat N}$
  is chosen such that $\hat N p_{\hat N} \to \infty$ for a
  sub-sequence, the statement is a direct consequence of
  proposition~\ref{p:MSE-fixed-N}.  Finally, if $\hat N p_{\hat N}$ is
  bounded, $\MSE\bigl( \hat Y_{\hat N}^{\delta_{\hat N}}\bigr)$ stays
  bounded away from~$0$ and thus
  \begin{equation*}
    \frac{\MSE\bigl(\hat Y_{\hat N}^{\delta_{\hat N}}\bigr)}
         {\cost\bigl(\hat Y_{\hat N}^{\delta_{\hat N}}\bigr)^{-4/(q+4)}}
    \longrightarrow \infty > 0
  \end{equation*}
  as $\hat N\to \infty$.  This completes the proof of the second
  statement.
\end{proof}


\section{Numerical Experiments}
\label{S:experiments}

To illustrate the results from section~\ref{S:results}, we present a
series of numerical experiments for the following toy problem.
\begin{itemize}
\item We choose $p=1$, and assume that our prior belief in the value
  of the single parameter $\theta$ is standard normally distributed.
\item We choose $d=2$, and assume the data $X$ to be composed of
  i.i.d.\ samples $X_1, X_2$ each with distribution $X_i \,|\, \theta
  \sim \CN(\theta , 1)$.
\item We choose $q=2$, and the (non-minimal) sufficient statistic to
  be $S(x) = x$ for all $x\in\R^2$.
\item We consider the test function
  $h(\theta)=1_{[-1/2,1/2]}(\theta)$, \textit{i.e.}\ the indicator
  function for the region $[-1/2,1/2]$. The ABC estimate is thus an
  estimate for the posterior probability $P\bigl( \theta\in [-1/2,1/2]
  \bigm| S=s^\ast \bigr)$.
\item We fix the observed data to be $s^*=(1,1)$.
\end{itemize}

This problem is simple enough that all quantities of interest can be
determined explicitly. In particular, $\theta|S$ is
$\CN\bigl((s_1+s_2)/3, 1/3\bigr)$ distributed, $\theta|S\!=\!s^*$ is
$\CN(2/3, 1/3)$ distributed, and $S$ is bivariate normally distributed
with mean~$0$ and covariance matrix
\begin{equation*}
  \Sigma = \begin{pmatrix}2&1\\1&2\end{pmatrix}.
\end{equation*}
Therefore, the prior and posterior expectation for $h(\theta)$ are
$\E\bigl(h(\theta)\bigr) = 0.3829$ and $\E\condl{h(\theta)}{S=s^\ast}
= 0.3648$, respectively.  Similarly, the constant from
lemma~\ref{l:bias} can be shown to be $C(s^*) = 0.0323$.

Assumption~A can be shown to hold. The function
\begin{equation*}
  \phi_1(s)
  = f_S(s)
  = \frac{1}{2\pi\sqrt{3}} e^{-\frac{1}{3} (s_1^2-s_1s_2+s_2^2)}
\end{equation*}
is multivariate normal, so its third derivatives exist, and are
bounded and continuous.  Similarly, the function
\begin{equation*}
  \phi_h(s)
  = \int_{-1/2}^{1/2} f_{\theta,S}(t,s) dt \leq \phi_1(s)
\end{equation*}
also has bounded and continuous third derivatives. Thus, the
assumptions hold.

The programs used to perform the simulations described in this
section, written using the R environment \citep{R} are available as
supplementary material.

\subsection*{Experiment 1}

Our first experiment validates the statement about the bias given in
lemma~\ref{l:bias}.  Since we know the exact posterior expectation,
the bias can be determined experimentally. For fixed $\delta$, we
generate $k$ independent ABC estimates, each based on $n$
proposals. For each of the $k$ estimates, we calculate its distance
from the true posterior expectation. We then calculate the mean and
standard error of these differences to obtain a Monte Carlo estimate
of the bias.

Repeating this procedure for several values of $\delta$, we can
produce a plot of the estimated bias against $\delta$, with $95\%$
error bars. Figure~\ref{fig:parabola} shows the result of a
simulation, using $n=500$ samples for each ABC estimate and $k=5000$
ABC estimates for each value of~$\delta$.  For comparison, the plot
includes the theoretically predicted asymptotic bias $C(s^\ast)
\delta^2$, using the value~$C(s^*) = 0.0323$.  The plot shows that for
small values of $\delta$ the theoretical curve is indeed a good fit to
the numerical estimates of the bias.  The result of the lemma is only
valid as $\delta \downto 0$ and indeed the plot shows a discrepancy
for larger values.  This discrepancy is caused by only a small
fraction of the sample being rejected; as $\delta$ increases, the
distribution of the ABC samples approaches the prior distribution.

\begin{figure}
  \begin{center}
    \includegraphics{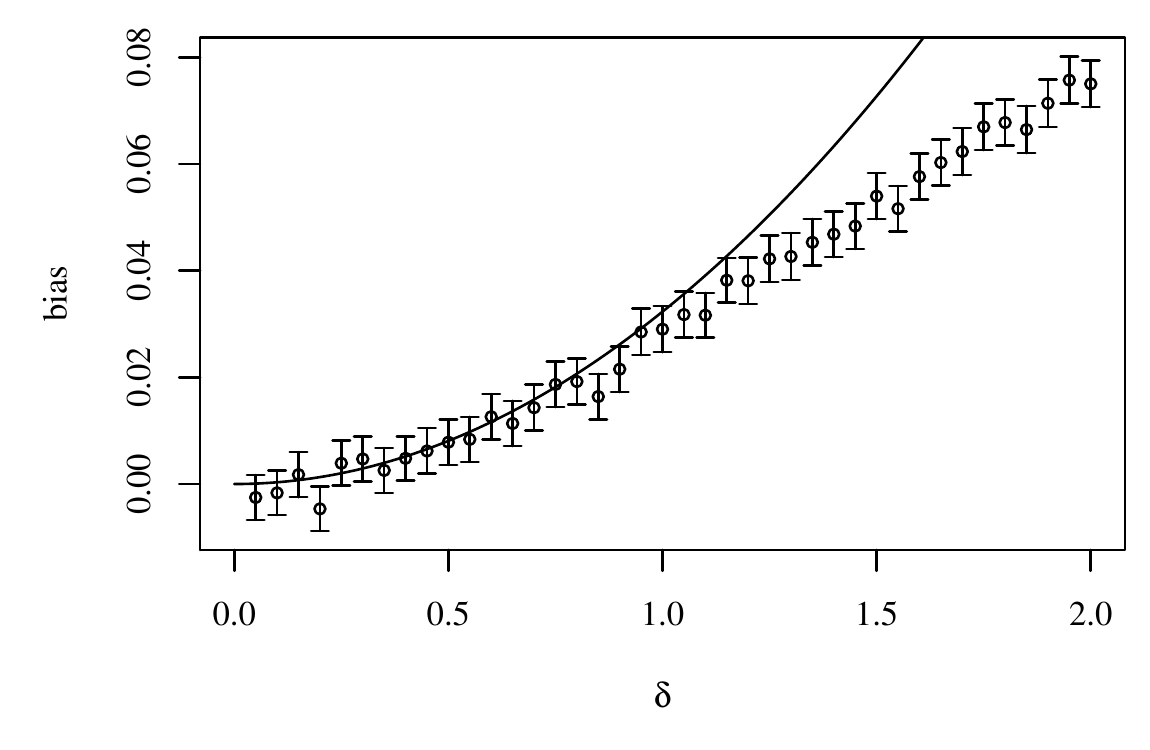}
  \end{center}
  \caption{\label{fig:parabola}Simulation results illustrating the
    relationship between bias and $\delta$. Circles give the mean
    empirical bias from $5000$ simulations for each value
    of~$\delta$.  The error bars indicate mean $\pm 1.96$ standard
    errors.  The solid line shows the theoretically predicted asymptotic
    bias from lemma~\ref{l:bias}.}
\end{figure}

\subsection*{Experiment 2}

Our second experiment validates the statement of theorem~\ref{t:rate},
by numerically estimating the optimal choice of delta and the
corresponding MSE.

For fixed values of expected computational cost and $\delta$, we
estimate the mean squared error by generating $k$ different ABC
estimates and taking the mean of their squared distance from the true
posterior expectation.  This reflects how the bias is estimated in
experiment~1. Repeating this procedure for several values of $\delta$,
the estimates of the MSE are plotted against~$\delta$.

Our aim is to determine the optimal value of $\delta$ for fixed
computational cost.  From lemma~\ref{l:cost} we know that the expected
cost is of order $n \delta^{-q}$ and thus we choose $n \sim \delta^2$
in this example.  From lemma~\ref{l:bias} we know that $\bias \sim
\delta^2$.  Thus, we expect the MSE for constant expected cost to be
of the form
\begin{equation}\label{eq:fitted}
  \MSE(\delta)
  = \frac{\Var}{n} + \bias^2
  = a\delta^{-2} +b\delta^4
\end{equation}
for some constants $a$ and~$b$.  Thus, we fit a curve of this form to
the numerically estimated values of the MSE.  The result of one such
simulation, using $k=500$ samples for each~$\delta$, is shown in
figure~\ref{fig:opt-delta}.  The curve fits the data well.

\begin{figure}
  \begin{center}
    \includegraphics{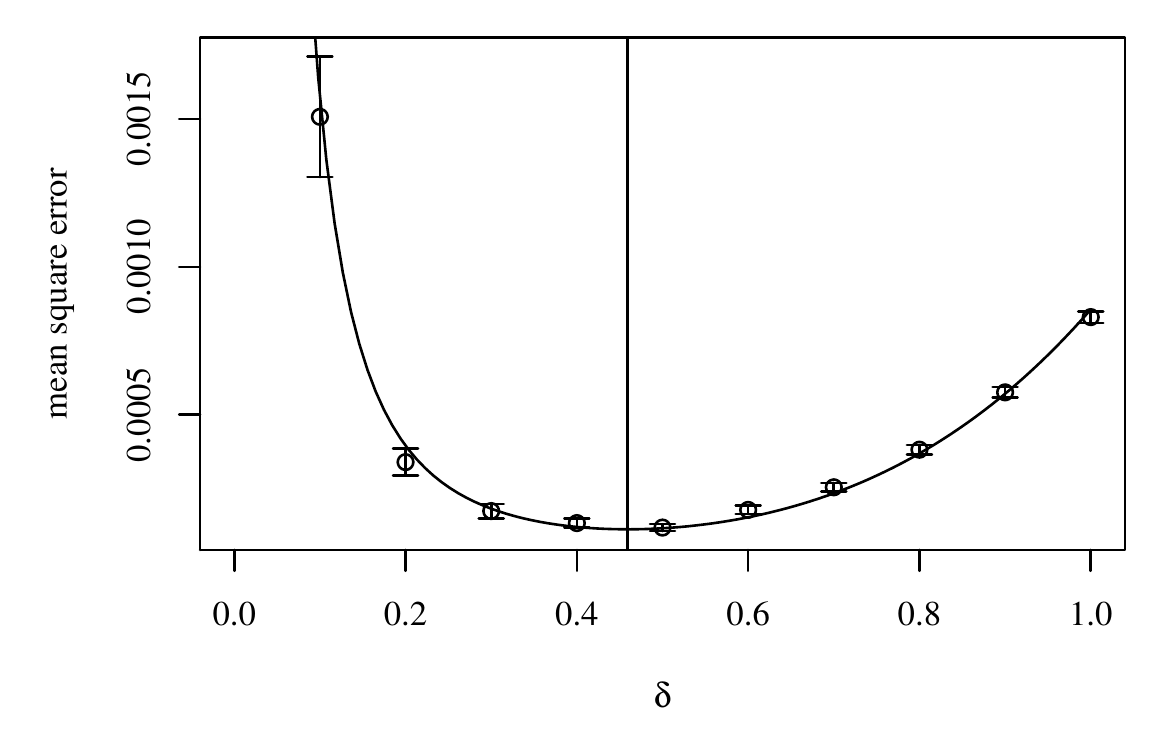}
  \end{center}
  \caption{\label{fig:opt-delta}The estimated MSE as a function of the
    tolerance~$\delta$ for fixed expected cost.
    The fitted curve has the form given in equation~\eqref{eq:fitted}.
    The location of the optimal~$\delta$ is marked by the vertical line.}
\end{figure}

We estimate the optimal values of $\delta$ and MSE, given an expected
computational cost, to be those at the minimum of the fitted curve.
Given the good fit between the predicted form~\eqref{eq:fitted} of the
curve and the empirical MSE values, this procedure promises to be a
robust way to estimate the optimal value of $\delta$.  The direct
approach would likely require a much larger number of samples to be
accurate.

Repeating the above procedure for a range of values of expected cost
gives corresponding estimates for the optimal values of $\delta$ and
the MSE as a function of expected cost.  We expect the optimal
$\delta$ and the MSE to depend on the cost like $x=A \cdot
\mathrm{cost}^B$.  To validate the statements of theorem~\ref{t:rate}
we numerically estimate the exponent~$B$ from simulated data.  The
result of such a simulation is shown in figure~\ref{fig:rates}.  The
data are roughly on straight lines, as expected, and the gradients are
close to the theoretical gradients, shown as smaller lines. The
numerical results for estimating the exponent~$B$ are given in the
following table.
\begin{center}
  \begin{tabular}{r|c|c|c}
    Plot&Gradient&Standard error&Theoretical gradient \\ \hline
    $\delta$	&$-0.167$	&$0.0036$	&$-1/6 \approx -0.167$ \\
    MSE	&$-0.671$	&$0.0119$	&$-2/3 \approx -0.667$
  \end{tabular}
\end{center}
The table shows an an excellent fit between empirical and
theoretically predicted values.

\begin{figure}
  \begin{center}
    \includegraphics{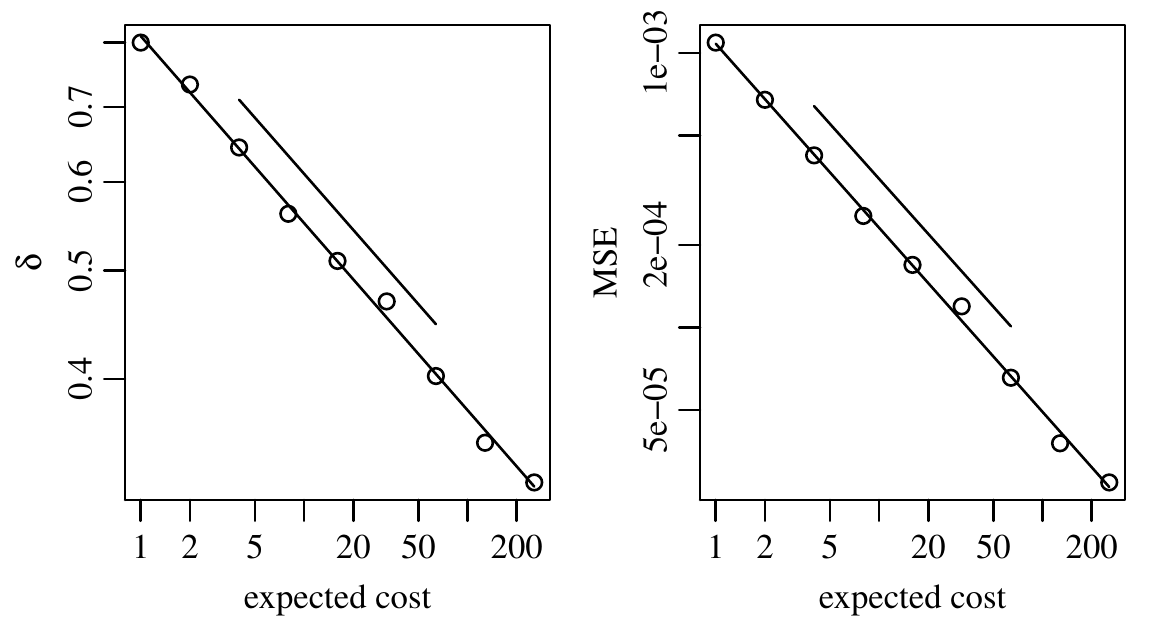}
  \end{center}
  \caption{\label{fig:rates}Estimated dependency of the optimal
    $\delta$ and of the corresponding MSE on the computational cost.
    The computational cost is given in arbitrary units, chosen such
    that the smallest sample size under consideration has
    cost~$1$. For comparison, the additional line above the fit has
    the gradient expected from the theoretical results.}
\end{figure}


\section{Discussion}
\label{S:discussion}

While the work of this article is mostly theoretical in nature, the
results can be used to guide the design of an analysis using ABC
methods.

Our results can be applied directly in cases where a pilot run is used
to tune the algorithm.  The performance of the ABC algorithm depends
both on the number~$n$ of samples accepted and on the
tolerance~$\delta$.  From theorem~\ref{t:rate} we know that optimality
is achieved by choosing $\delta$ proportional to $n^{-1/4}$.
Consequently, if the full run is performed by increasing the number of
accepted ABC samples by a factor of~$k$, then the tolerance~$\delta$
for the full run should be divided by~$k^{1/4}$.  In this case the
expected running time satisfies
\begin{equation*}
  \mathrm{cost} \sim n\delta^{-q} \sim n^{(q+4)/4}
\end{equation*}
and, using lemma~\ref{l:cost}, we have
\begin{equation*}
  \mathrm{error}
  \sim \mathrm{cost}^{-2 / (q+4)}
  \sim n^{-1/2}.
\end{equation*}
There are two possible scenarios:
\begin{itemize}
\item If we want to reduce the root mean-squared error by a factor
  of~$\alpha$, we should increase the number $n$ of accepted ABC
  samples by $\alpha^2$ and reduce the tolerance~$\delta$ by a factor
  of~$(a^2)^{1/4} = \sqrt{\alpha}$.  These changes will increase the
  expected running time of the algorithm by a factor
  of~$\alpha^{(q+4)/2}$.
\item If we plan to increase the (expected) running time by a factor
  of~$\beta$, we should increase the number of accepted samples by a
  factor of~$\beta^{4/(q+4)}$ and divide $\delta$
  by~$\beta^{1/(q+4)}$.  These changes will reduce the root mean
  squared error by a factor of~$\beta^{2/(q+4)}$.
\end{itemize}
These guidelines will lead to a choice of parameters which, at least
asymptotically, maximises the efficiency of the analysis.


\vskip2\baselineskip
\noindent
\textbf{Acknowledgements.} The authors thank Alexander Veretennikov
for pointing out that the Lebesgue differentiation theorem can be used
in the proof of proposition~\ref{p:convergence}.  MW was supported by
an EPSRC Doctoral Training Grant at the School of Mathematics,
University of Leeds.

\bibliographystyle{plainnat}
\bibliography{references}

\end{document}